\documentclass[12pt,british]{article}
\usepackage{color}
\usepackage{babel}
\usepackage{amsmath,amsthm}
\usepackage{amsfonts}
\usepackage{amssymb}
\usepackage{graphicx}
\usepackage[utf8]{inputenc}
\usepackage{multicol}
\usepackage[all]{xy}
\usepackage{bbm,enumerate}
\usepackage{geometry}
\newtheorem{theorem}{Theorem}

\newtheorem{corollary}[theorem]{Corollary}

\newtheorem{lemma}[theorem]{Lemma}
\newtheorem{proposition}[theorem]{Proposition}

\theoremstyle{definition}

\newtheorem{definition}[theorem]{Definition}

\newtheorem{remark}[theorem]{Remark}

\usepackage[
bookmarks=true,
breaklinks=true,
bookmarksnumbered = true,
colorlinks= true,
urlcolor= green,
anchorcolor = yellow,
citecolor=blue,
]{hyperref}                   

\begin{document}
	
	\title{Computing the one-parameter Nielsen number for homotopies on the n-torus}
	\author{WESLEM LIBERATO SILVA
		~\footnote{Departamento de Ciências Exatas e Tecnológicas, Universidade Estadual de Santa Cruz, Rodovia Jorge Amado, Km 16, Bairro Salobrinho, CEP 45662-900, Ilhéus-BA, Brazil.
			e-mail: \texttt{wlsilva@uesc.br}}
	}
	%
	\maketitle

\begingroup
\renewcommand{\thefootnote}{}
\endgroup

\begin{abstract}
Let $F: \mathbb{T}^{n} \times I \to \mathbb{T}^{n}$ be a homotopy on a n-dimensional torus. The main purpose of this paper is to present a formula for the one-parameter Nielsen number $N(F)$  of $F$ in terms of its induced homomorphism. If $L(F)$ is the one-parameter Lefschetz class of $F$ then $L(F)$ is given by $L(F) = \ N(F)\alpha,$ for some $\alpha  \in  H_{1}(\pi_{1}(\mathbb{T}^{n}),\mathbb{Z}).$ 
\end{abstract}

\maketitle


\section{Introduction}\label{introduction}
Let $F: X \times I \to X$ be a homotopy on a finite CW complex  $X$ and $G = \pi_{1}(X,x_{0})$. Here $I$ will denote the unit interval.  We say that $(x,t) \in X \times I$ is a fixed point of $F$ if $F(x,t) = x.$ We denote the fixed points set of $F$ by $Fix(F).$ R. Geoghegan and A. Nicas in \cite{G-N-94} developed a one-parameter theory and defined the one-parameter trace $R(F)$ of $F$ to study the fixed points set of $F$. From trace $R(F)$ we define the one-parameter Nielsen number $N(F)$ of $F$ and the one-parameter Lefschetz class $L(F).$ These invariants are computable, depending only on the homotopy class of $F$ relative to $X \times \{0, 1\}.$

The study of the fixed points of a homotopy has been considered by many authors, see for example \cite{H-83}, \cite{D-94} and \cite{G-K-02}. Here is important to point that only the reference \cite{D-94} uses the approach developed in \cite{G-N-94}. 
Following \cite{G-N-94} we have an important application of the trace $R(F).$  Given a smooth flow $\Psi: M \times \mathbb{R} \to M$ on a closed oriented manifold one may regard any finite portion of $\Psi$ as a homotopy. Write $F = \Psi|: M \times [a, b] \to M.$ The traces L(F) and R(F) recognize dynamical meaning of $\Psi.$ When  $a > 0,$ L(F) detects the Fuller homology class, derived from Fuller’s index theory, see \cite{F-67}. Thus is possible to study periodic orbits of $\Psi$ using the one-parameter theory, see \cite{G-N-94-2}. 

The result of this paper allows as to solve the important problem which is the calculation of periodic orbits of a flow on the n-torus. In fact, given a smooth flow $\Psi: \mathbb{T}^{n} \times \mathbb{R} \to \mathbb{T}^{n}$ on n-torus we write $F = \Psi|: \mathbb{T}^{n} \times [a, b] \to \mathbb{T}^{n}$ for a finite portion of $\Psi.$ In the case $n=2,$ in \cite[Example 5.10, pg 431]{G-N-94}, was presented an example of calculation of periodic orbits. In this paper we prove that the Lefschetz class $L(F)$ of $F$ is given by $L(F) =  N(F)\alpha,$ for some $\alpha  \in H_{1}(\pi_{1}(\mathbb{T}^{n}),\mathbb{Z}),$ and we present a formula for $N(F).$

Let $\mathbb{T}^{n} = {\mathbb{R}^{n}}/{\mathbb{Z}^{n}}$ be the $n$-torus and $v = [(0,0, ... ,0)].$ We denote $$\pi_1(\mathbb{T}^{n},v) = \langle u_1,u_2, ... ,u_n|u_iu_ju_i^{-1}u_j^{-1}=1, \,\,\textrm{for all} \,\, i \neq j \rangle.$$ 

We say that a homotopy $H: \mathbb{T}^{n} \times I \to \mathbb{T}^{n}$ is affine if there exist $H^{'}: \mathbb{R}^{n} \times I \to \mathbb{R}^{n}$ such that $H \circ (p_{\mathbb{T}^{n}} \times Id) = p_{\mathbb{T}^{n}} \circ H^{'},$ where $p_{\mathbb{T}^{n}}: \mathbb{R}^{n} \to \mathbb{T}^{n}$ is the natural projection and $H^{'}$ is given by
$$H^{'}(x_1, \cdots, x_n, t) = ( \displaystyle \sum_{j=1}^{n} a_{1j}x_j + c_1t + \epsilon_1 , \cdots , \displaystyle \sum_{j=1}^{n} a_{nj}x_j + c_nt + \epsilon_n ),$$
for some $a_{ij}, c_i \in \mathbb{Z}$ and $ 0 \leq \epsilon_i < 1.$

Given $F: \mathbb{T}^{n} \times I \to \mathbb{T}^{n}$ a homotopy we denote by $w = F(v,I).$ We will assume that the homotopy class of $F$ relative to $\mathbb{T}^{n} \times \{0,1\}$ contains one affine homotopy where the $\epsilon_i$ are chosen such that $F$ has no fixed points in $ \mathbb{T}^{n} \times \{0,1\}$ when the classical Nielsen number $N(F|_{\mathbb{T}^{n}})$ is zero. From this hypothesis follows that  $w$ is a loop in $\mathbb{T}^{n}.$ We can write 
$$[w] = u_1^{c_1}u_2^{c_2}...u_n^{c_n}$$ 
for some integers $c_1, c_2, ... , c_n.$  
Let $\phi$ be the homomorphism given by the following composition: $$\pi_1(\mathbb{T}^{n} \times I, (v,0)) \stackrel{F_{\#}}{\to} \pi_1(\mathbb{T}^{n}, F(v,0)) \stackrel{c^{-1}_{[\tau]}}{\to} \pi_1(\mathbb{T}^{n},v),$$
where $\tau$ is the path in $\mathbb{T}^{n}$ from $v$ to $F(v,0)$ and $c_{[\tau]}$ is the isomorphism that changes the base point. 
Suppose that the Nielsen number of $F$ restricted to $\mathbb{T}^{n},$ $N(F|_{\mathbb{T}^{n}}) = |det([\phi]-I)|,$ is zero. Let $w_1$ be an eigenvector of $[\phi]$ associated to $1.$ Complete $\{w_1, w_2, ... , w_n\}$  for a basis of $\mathbb{R}^{n}.$ In this new basis the matrix of $\phi$ is given by: 
$$[\phi] = 
\left( \begin{array}{cccc}
	1 & b_{12} & \cdots & b_{1n}	\\
	0 & b_{22} &  \cdots & b_{2n}   \\
	\vdots & \vdots   &    &  \vdots \\
	0  & b_{n2}	&  \cdots &  b_{nn} \\	
\end{array}\right).$$
If $P: \mathbb{T}^{n} \times I \to \mathbb{T}^{n}$ is the projection then 
$[\phi] - [P_{\#}] = 
\left( \begin{array}{cccc}
	0 & b_{12} & \cdots & b_{1n}	\\
	0 & b_{22}-1 &  \cdots & b_{2n}   \\
	\vdots & \vdots   &    &  \vdots \\
	0  & b_{n2}	&  \cdots &  b_{nn}-1 \\	
\end{array}\right).$
We denote $$A = \left( \begin{array}{cccc}
	b_{12} & \cdots & b_{1n} & c_1	\\
	b_{22}-1 &  \cdots & b_{2n}  & c_2 \\
	\vdots   &    &  \vdots  & \vdots  \\
	b_{n2}	&  \cdots &  b_{nn}-1 & c_n \\	
\end{array}\right).$$

With the above hypothesis and notations we present the main result of this paper.

\begin{theorem}  \label{main-theorem-1}
	Given a homotopy $F: \mathbb{T}^{n} \times I \to \mathbb{T}^{n}$ then the one-parameter Lefschetz class of $F$ is given by:
	$$L(F) =  N(F) \alpha,$$
	where $N(F)$ is the one-parameter Nielsen number of $F$ and $\alpha$ is a  class in $H_1(\pi_1(\mathbb{T}^{n}), \mathbb{Z}).$ The one-parameter Nielsen number of $F$ is given by:
	\begin{equation}  \label{main-equation-1}
		N(F) = \left\{ \begin{array}{lllll}  
			|det(A)| &   \,\,\, if & N(F|_{\mathbb{T}^{n}}) = 0,  \\
			&  &    \\
			0 &  &\textrm{otherwise}   . \\
		\end{array}\right.
	\end{equation}
\end{theorem}	

%
%

When $n=1$ the Theorem \ref{main-theorem-1} also was proved in \cite[Theorem 5.1]{G-N-94} where the statement is written in a slightly different form. In \cite{S-20} only the first part of  Theorem \ref{main-theorem-1} was proved for the case $n=2$, that is, it was proved that $L(F) = \pm N(F)\alpha$ for any homotopy $F: \mathbb{T}^{2} \times I \to \mathbb{T}^{2},$ however a formula for the one-parameter Nielsen number has not been presented. Some computations of $N(F)$ when $n=2$ were presented in \cite{S-14}. In this work we generalize \cite{S-20} and present a formula for $N(F)$ for any homotopy $F$ on the $n$-torus. The results of this work is in some sense a version of the main result of  \cite{B-B-P-T-75} for the one-parameter case.  

This paper is organized into five sections. In Section 2 we present a brief review of the one-parameter fixed point theory.
In Section 3 we study some properties of the semiconjugacy classes on the $n$-torus. In Section 4 we present the proof of Theorem \ref{main-theorem-1}. The Section 5 is dedicated to presenting some  applications of Theorem \ref{main-theorem-1}.


\section{One-parameter Fixed Point Theory} \label{section-2}

To facilitate the reading of this paper we will do in this section a brief review of definition of the one-parameter trace for a homotopy $F: X \times I \to X$ where $X$ is a finite CW complex and $F$ is cellular. For more details see \cite{G-N-94}.

Let $R$ be a ring and $M$ an $R-R$ bimodule, that is, a left and right R-module satisfying $(r_{1}m)r_{2} = r_{1}(mr_{2})$ 
for all $m \in M$, and $r_{1},r_{2} \in R$. The Hochschild chain complex $\{C_{\ast}(R, M) ,d \}$ is given by 
$C_{n}(R,M) = R^{\otimes n} \otimes M$ where $R^{\otimes n}$ is the tensor product of n copies of $R$, taken 
over the integers, and  
$$
\begin{array}{lcl}
	d_{n}(r_{1} \otimes \ldots \otimes r_{n} \otimes m ) & = &
	r_{2} \otimes \ldots \otimes r_{n} \otimes m r_{1}  \\
	&  & + \displaystyle \sum_{i=1}^{n-1} (-1)^{i} r_{1} \otimes \ldots \otimes 
	r_{i} r_{i+1} \otimes \ldots \otimes r_{n} \otimes m  \\
	&  & + \, (-1)^{n} r_{1} \otimes \ldots \otimes r_{n-1} \otimes r_{n} m . \\
\end{array}
$$

The n-th homology of this complex is the Hochschild homology of $R$ with coefficient bimodule $M,$ it 
is denoted by $HH_{n}(R,M)$. There are other ways of presenting the definition of Hochschild homology, for example see \cite{N-05}.

In the particular cases $n=1,2$ we have the formula      
$ d_{2}(r_{1} \otimes r_{2}\otimes m) = r_{2}\otimes m r_{1} -  r_{1}r_{2}\otimes m + r_{1}\otimes r_{2}m$ and  
$ d_{1}(r \otimes m) = mr - rm $. Using the expression of $d_2$ and the $2$-chain $1 \otimes 1 \otimes m$ we 
obtain;
 
\begin{lemma} \label{lemma-unit}
If $1 \in R$ is the unit element and $m \in M$ then the 1-chain $1 \otimes m$ is a boundary.
\end{lemma}

For the definition of $R(F)$ we use the Hochschild homology in the following situation:  Let $G$ be a group and $\phi:G \to G$ an endomorphism. 
Also denote by $\phi$ the induced ring homomorphism $\mathbb{Z}G \to \mathbb{Z}G$. 
Take the ring $R = \mathbb{Z}G$ and $M = (\mathbb{Z}G)^{\phi}$ the $\mathbb{Z}G-\mathbb{Z}G$ bimodule whose underlying 
abelian group is $\mathbb{Z}G$ and the bimodule structure is given by $g.m = gm$ and  $m.g = m \phi(g)$.

We say that two elements $g_{1},g_{2}$ in $G$ are semiconjugate if there exists $g \in G$ such that $g_{1} = g g_{2} \phi(g^{-1})$.
We write $C(g)$ for the semiconjugacy class containing $g$ and $G_{\phi}$ for the set of semiconjugacy classes. 
Thus, we can decompose $G$ in the union of its semiconjugacy classes. 
This partition induces a direct sum decomposition of $HH_{\ast}(\mathbb{Z}G,(\mathbb{Z}G)^{\phi}).$

Note that each generating chain $\gamma = g_{1} \otimes ... \otimes g_{n} \otimes m $ can be written in canonical form 
as $g_{1} \otimes ... \otimes g_{n}  \otimes g_{n}^{-1}... g_{1}^{-1}g $ where $g = g_{1} ... g_{n}m.$ We will say that $g$ ``marks'' a semiconjugacy class. The decomposition {\small $(\mathbb{Z}G)^{\phi} \cong \bigoplus_{C \in G_{\phi}} \mathbb{Z}C$} as a direct 
sum of abelian groups determines a decomposition of chains complexes $C_{\ast}(\mathbb{Z}G,(\mathbb{Z}G)^{\phi}) \cong 
\bigoplus_{C \in G_{\phi}} {C_{\ast}(\mathbb{Z}G,(\mathbb{Z}G)^{\phi})}_{C} $ where  ${C_{\ast}(\mathbb{Z}G,(\mathbb{Z}G)^{\phi})}_{C} $ 
is the subgroup of $C_{\ast}(\mathbb{Z}G,(\mathbb{Z}G)^{\phi})$ generated by those generating chains whose markers lie in $C$. 
Therefore, we have the following isomorphism: $HH_{\ast}(\mathbb{Z}G,(\mathbb{Z}G)^{\phi}) \cong 
\bigoplus_{C \in G_{\phi}} {HH_{\ast}(\mathbb{Z}G,(\mathbb{Z}G)^{\phi})}_{C} $ where the summand ${HH_{\ast}(\mathbb{Z}G,(\mathbb{Z}G)^{\phi})}_{C} $ 
corresponds to the homology classes marked by the elements of $C$. This summand is called the $C-$component.

Let $Z(h) = \{g \in G| h = gh\phi(g^{-1})  \}$ be the semicentralizer of $h \in G$. Choosing representatives $g_{C} \in C$, then we have 
the following proposition whose proof is in \cite{G-N-94}.
\begin{proposition} \label{prop-decomposition}
	Choosing representatives $g_{C} \in C$ then we have 
	$$HH_{\ast}(\mathbb{Z}G,(\mathbb{Z}G)^{\phi}) \cong 
	\bigoplus_{C \in G_{\phi}} {H_{\ast}(Z(g_{C}))}_{C} $$
	where ${H_{\ast}(Z(g_{C}))}_{C} $ corresponds to the summand ${HH_{\ast}(\mathbb{Z}G,(\mathbb{Z}G)^{\phi})}_{C}$. 
\end{proposition}

\begin{lemma} \label{lemma-se}
	If $G = \pi_1(X, v)$ is an abelian group then the cardinality of semiconjugacy classes in $G$ is the cardinality of $coker(\phi-P_{\#})$ in $G,$ where $P: X \times I \to X$ is the projection.  
\end{lemma}
\begin{proof}
	In fact, two elements $g_1$ and $g_2$ in $G$ belong to the same semiconjugacy class if and only if there exists $g \in G$ such that $g_1 = g g_2 \phi(g^{-1}).$ This is equivalent to $g_2 - g_1 = \phi(g) - P_{\#}(g),$ 
	because $G$ is abelian. On the other hand, the last equation is equivalent to say that $g_1$ and $g_2$ belong the same class in $coker(\phi-P_{\#})$ in $G.$ 
\end{proof}

\subsection{One-parameter trace $R(F)$} \label{one-parameter-theory}

Let $X$ be a finite connected CW complex and $F: X \times I \to X$ a cellular homotopy.
We consider  $I = [0, 1]$ with the usual CW structure and orientation of cells, 
and $X \times I$ with the product CW structure, where its cells are given the product orientation.
Pick a basepoint $(v,0) \in X \times I$, and a basepath $\tau$ in $X$ from $v$ to $F(v,0)$. 
We identify $\pi_{1}(X \times I, (v,0)) \equiv G$ with $\pi_{1}(X,v) $ via the isomorphism 
induced by projection $p: X \times I \to X$. We write $\phi: G \to G$ for the homomorphism;
$$ \pi_{1}(X \times I, (v,0)) \stackrel{F_{\#}}{\to} \pi_{1}(X, F(v,0)) \stackrel{c_{\tau}}{\to} \pi_{1}(X, v) $$

For each cell $E$  in $X$, we choose a lift $\tilde{E}$ in the universal cover $\tilde{X}$ and we orient $\tilde{E}$ compatibly with $E$. Let $\tilde{\tau}$ be the lift of the basepath $\tau$ which starts in the basepoint $\tilde{v} \in \tilde{X}$ 
and $\tilde{F}: \tilde{X} \times I \to \tilde{X}$ the unique lift of $F$ satisfying $\tilde{F}(\tilde{v},0) = \tilde{\tau}(1)$.
We can regard $C_{\ast}(\tilde{X})$ as a right $\mathbb{Z}G$ chain complex as follows: if $\omega$ is a loop 
at $v$ which lifts to a path $\tilde{\omega}$ starting at $\tilde{v}$ then $\tilde{E}[\omega]^{-1} = h_{[w]}(\tilde{E})$, 
where $h_{[\omega]} $ is the covering transformation sending $\tilde{v}$ to $\tilde{\omega}(1)$.
The homotopy $\tilde{F}$ induces a chain homotopy $\tilde{D_{k}}: C_{k}(\tilde{X}) \to C_{k+1}(\tilde{X})$ 
given by $$\tilde{D_{k}}(\tilde{E}) = (-1)^{k+1}\tilde{F_{k}}(\tilde{E} \times I) \in C_{k+1}(\tilde{X}),$$ 
for each cell $\tilde{E} \in \tilde{X}$. 
This chain homotopy satisfies; $\tilde{D}(\tilde{E}g) = \tilde{D}(\tilde{E}) \phi(g)$ and the boundary operator 
$\tilde{\partial_{k}}: C_{k}(\tilde{X}) \to C_{k-1}(\tilde{X})$ satisfies; $\tilde{\partial}(\tilde{E}g)= \tilde{\partial}(\tilde{E})g$.
Define endomorphism of $\oplus_{k} C_{k}(\tilde{X})$ by $\tilde{D_{\ast}} = \oplus_{k} (-1)^{k+1} \tilde{D_{k}}$,
$ \tilde{\partial_{\ast}} = \oplus_{k} \tilde{\partial_{k}}$, $\tilde{F_{0 \ast}} = \oplus_{k} (-1)^{k} \tilde{F_{0 k}} $ 
and $\tilde{F_{1 \ast}} = \oplus_{k} (-1)^{k} \tilde{F_{1 k}} $.
We consider  trace$(\tilde{\partial_{\ast}} \otimes \tilde{D_{\ast}}) \in HH_{1}(\mathbb{Z}G, (\mathbb{Z}G)^{\phi})$. 
This is a Hochschild 1-chain whose boundary is;
trace$(\tilde{D_{\ast}}\phi(\tilde{\partial_{\ast}}) - \tilde{\partial_{\ast}} \tilde{D_{\ast}}) .  $
We denote by $G_{\phi}(\partial(F))$ the subset of $G_{\phi}$ consisting of semiconjugacy classes associated to fixed 
points of $F_{0}$ or $F_{1}$.

\begin{definition} \label{definition-trace}
	The  one-parameter trace of homotopy $F$ is:
	$$R(F) \equiv T_{1}(\tilde{\partial_{\ast}} \otimes \tilde{D_{\ast}}; G_{\phi}(\partial(F))) \in 
	\bigoplus_{C \in G_{\phi} - G_{\phi}(\partial(F))} HH_{1}(\mathbb{Z}G, (\mathbb{Z}G)^{\phi})_{C}  $$
	$$ \cong \bigoplus_{C \in G_{\phi} - G_{\phi}(\partial(F))} H_{1}(Z(g_{C})). $$
\end{definition}

Let $(x,t)$ and $(y,s)$ be two points in $Fix(F).$ We say that these points are in the same fixed point class if there exists a path $\gamma: I \to X \times I$ with $\gamma(0) = (x,t),$ $\gamma(1) = (y,s)$ and $(P \circ \gamma)(F \circ \gamma)^{-1}$ is homotopically trivial. Here $P: X \times I \to X$ is the projection. This defines an equivalence relation $\sim$ on $Fix(F).$   
The function $\Psi: Fix(F)/ \sim \,\,\, \to G_{\phi}$ defined by $\Psi([(x,t)]) = [(P\circ \nu)(F \circ \nu)^{-1} \circ \tau^{-1}]$ is injective, where $\nu$ is any path from base point $(v,0)$ to $(x,t).$ 

Supposing that $F$ is transverse the projection $P: X \times I \to X$ then $Fix(F)$ is composed by circles in $X \times (0,1)$ and arcs connecting $X \times \{0,1\},$ see \cite{G-N-94}. By Definition \ref{definition-trace} we are only interested in the circles in $X \times (0,1)$ because all semiconjugacy classes associated to fixed point classes that intersect $X \times \{0,1\}$ has no contribution  to the expression of $R(F).$ From \cite{D-94} one can always reduce to the case in which only one circle occurs in each fixed point class.

\begin{definition} Let $K$ a fixed point class of $F$ and $ \Psi(K) = C\in G_{\phi}.$ The one-parameter fixed point index of $K$ is the $C$-component of $R(F),$ $i(F,C),$ in ${HH_{1}(\mathbb{Z}G, (\mathbb{Z}G)^{\phi})}_{C}.$  The one-parameter fixed point index $i(F,C)$ is zero if $i(F,C)$ is the trivial homology class. 
\end{definition}

\begin{definition} Given a cellular homotopy $F: X \times I \to X$ the  one-parameter Nielsen number, $N(F)$, of $F$ 
	is the number of components $i(F,C)$ with nonzero fixed point index.
\end{definition}

\begin{definition} The  one-parameter Lefschetz class, $L(F)$, of $F$ is defined by;  
	$$L(F) = \displaystyle \sum_{C \in G_{\phi} - G_{\phi}(\partial F)} j_{C}(i(F,C)) $$ where 
	$j_{C}: H_{1}(Z(g_{C})) \to H_{1}(G)$ is induced by the inclusion $Z(g_{C}) \subset G$.
\end{definition}

\begin{remark} From \cite[Theorem 1.9 item c]{G-N-94}, to compute the one-parameter trace $R(F)$ of $F: X \times I \to X$  is enough compute $R(F')$ for $F'$ where $F'$ is a map homotopic to $F,$ relative to $X \times \{0,1\},$ which is cellular. 
\end{remark}


\section{Semiconjugacy classes on n-torus} \label{section-3}

In this section we describe some results about the semiconjugacy classes on a n-torus related to a homotopy $F: \mathbb{T}^{n} \times I \to \mathbb{T}^{n}.$

Let $\mathbb{T}^{n} = {\mathbb{R}^{n}}/{\mathbb{Z}^{n}}$ be the $n$-torus and $v = [(0,0, ... ,0)].$ We denote $$G = \pi_1(\mathbb{T}^{n},v) = \langle u_1,u_2, ... ,u_n|u_iu_ju_i^{-1}u_j^{-1}=1, \,\,\, \textrm{for all} \,\, i \neq j \rangle.$$ Given $F: \mathbb{T}^{n} \times I \to \mathbb{T}^{n}$ a homotopy, where $I$ is the unit interval, we denote by $w = F(v,I)$ the path in $\mathbb{T}^{n}.$ Assume that $w$ is a loop in $\mathbb{T}^{n}.$ Therefore we can write 
$$[w] = u_1^{c_1}u_2^{c_2}...u_n^{c_n}$$ 
for some integers $c_1, c_2, ... , c_n.$  
Let $\phi$ be the homomorphism given by the following composition: $$\pi_1(\mathbb{T}^{n} \times I, (v,0)) \stackrel{F_{\#}}{\to} \pi_1(\mathbb{T}^{n}, F(v,0)) \stackrel{c^{-1}_{[\tau]}}{\to} \pi_1(\mathbb{T}^{n},v),$$
where $\tau$ is a base path from $v$ to $F(v,0).$ 

Let us consider the isomorphism $\Theta : G = \pi_1(\mathbb{T}^{n},v) \to \mathbb{Z}^{n}$ defined by $\Theta(u_1^{k_1} \cdots u_n^{k_n}) = (k_1, \cdots, k_n).$ By abuse of notation we will sometimes write $\Theta(g) \equiv g.$

Two elements  $g_{1}$ and  $ g_{2} $ in $G$ belong to the same semiconjugacy class if, and only if, there exists $g \in G$ such that $g_1 = g g_2 \phi(g^{-1}),$ in this case this is equivalent to saying; 
$$(\phi - P_{\#})(\Theta(g)) = \Theta(g_2) - \Theta(g_1),$$ 
where $P: \mathbb{T}^{n} \times I \to \mathbb{T}^{n}$ is the projection. Thus we have:

\begin{lemma} \label{lemma-semi} For each $g \in G$ the semicentralizer $Z(g)$ is isomorphic to the kernel of $(\phi-P_{\#}).$
\end{lemma}
\begin{proof}
It follows from the definition of $Z(g)$ given on page 3. 
\end{proof}

\begin{proposition} \label{prop-nonzero}
	Let $F: \mathbb{T}^{n} \times I \to \mathbb{T}^{n}$ be a homotopy. If  the Nielsen number of $F$ restricted to $\mathbb{T}^{n}$ is nonzero then $R(F) = 0,$ which implies $L(F) = 0$ and $N(F) = 0.$
\end{proposition}
\begin{proof}
	If $N(F|_{\mathbb{T}^{n}}) \neq 0$ then by \cite{B-B-P-T-75} we have $|det([\phi] - I)| \neq 0.$ From Lemma \ref{lemma-semi} the semicentralizer $Z(g)$ is trivial for all $g$ in $G.$ Thus $H_1(Z(g_{C}))$ is trivial for each $g_{C}$ which represents a semiconjugacy class $C.$ By decomposition presented in Section \ref{section-2} we must have $HH_1(\mathbb{Z}G, (\mathbb{Z}G)^{\phi})=0.$ 
	Therefore, we obtain $R(F) = 0,$ which implies $L(F) = 0$ and $N(F) = 0.$	
\end{proof}
Note that in the situation of Proposition \ref{prop-nonzero} the cardinality of $G_{\phi}$ is infinite.

From now on, we will assume that the classical Nielsen number of $F: \mathbb{T}^{n} \times I \to \mathbb{T}^{n}$ restricted to $\mathbb{T}^{n}$ is zero, that is, $|det([\phi]- I)| = 0.$ Let $w_1$  a eigenvector of $[\phi]$ associated to 1. Complete $\{w_1, w_2, ... , w_n\}$  for a basis of $ \mathbb{R}^{n}.$  With respect to  this new base the matrix of $[\phi]$ has the following expression:
$$[\phi] = 
\left( \begin{array}{cccc}
	1 & b_{12} & \cdots & b_{1n}	\\
	0 & b_{22} &  \cdots & b_{2n}   \\
	\vdots & \vdots   &    &  \vdots \\
	0  & b_{n2}	&  \cdots &  b_{nn} \\	
\end{array}\right).$$

We will assume from now on that $[\phi]$ has the above expression. Also we denote $B = u_1^{k_1} ... u_n^{k_n}$ and $D = u_1^{l_1}...u_n^{l_n}$  elements in $G,$ where $k_j, l_j \in \mathbb{Z},$ for all $1 \leq j \leq n.$

\begin{lemma} \label{lemma-1}
	The 1-chain  $B \otimes D$ is a cycle in $HH_1(\mathbb{Z}G, (\mathbb{Z}G)^{\phi})$ if, and only if, the element $(k_1,...,k_n) \in \mathbb{Z}^{n} $ belongs to the kernel of $([\phi]-I)$. Therefore,
	if $rank([\phi] - I)= n-1$ then $B \otimes D$ is a cycle if, and only if, $k_2 = ... = k_n = 0.$
\end{lemma}
\begin{proof} 
In fact, the 1-chain $B \otimes D$ is a cycle if and only if  $ d_1(B \otimes D) = 0,$ that is, if and only if $0 = D\phi(B)- BD.$ Since $G$ is abelian then this is equivalent $(\phi-I)(B) = 0.$ The last equation is equivalent to say that 
	$(k_1, ... , k_n) \in ker([\phi]-I).$ In other words 
 $([\phi]-I)(B) = 0$ is equivalent to
	$$\left( \begin{array}{cccc}
		0 & b_{12} & \cdots & b_{1n}	\\
		0 & b_{22}-1 &  \cdots & b_{2n}   \\
		\vdots & \vdots   &    &  \vdots \\
		0  & b_{n2}	&  \cdots &  b_{nn}-1 \\	
	\end{array}\right) 
	\left(\begin{array}{c}
		k_1 \\
		k_2 \\
		\vdots \\
		k_n \\
	\end{array}
	\right) = 0 .$$
	
Thus if $rank([\phi]-I) = n-1$ then the system of equations above implies $k_2 = ... = k_n = 0,$ and therefore the 1-cycle $B \otimes D$ is written as $u_1^{k_1} \otimes D.$
\end{proof}

Let $E = u_1^{d_1} \cdots u_n^{d_n}.$
Given a 2-chain $ B \otimes D \otimes E \in C_{2}(\mathbb{Z}G, (\mathbb{Z}G)^{\phi})$ then
$$\begin{array}{l}
	d_{2}(B \otimes D \otimes E) 
	= D \otimes E\phi(B) - BD \otimes E + B \otimes DE. \\
\end{array}$$ 

The above expression will be used in the proof of the next result.

\begin{proposition} \label{prop-1}
	The 1-chain $u_1^{k_1} \otimes D \in C_{1}(\mathbb{Z}G,(\mathbb{Z}G)^{\phi})$ is homologous to the 1-chain $k_1 u_1 \otimes u_1^{k_1-1}D$ for all $k_1 \in \mathbb{Z}$. 
\end{proposition}
\begin{proof}
	For $k_1 = 1$ the proposition is clearly true. For $k_1 = 0$ the result is a consequence of Lemma \ref{lemma-unit}. Let us assume that for some $s > 0$ the 1-chain $u_1^{s} \otimes D$  is homologous to $ s u_1 \otimes u_1^{s-1} D$ for any $D$ in $G.$  
	Taking the 2-chain,
	$u_1^{s} \otimes u_1 \otimes D \in C_{2}(\mathbb{Z}G,(\mathbb{Z}G)^{\phi}),$ we obtain  
	{ \small $$\begin{array}{lll}
			d_{2}(u_1^{s} \otimes u_1 \otimes D ) 
			& = &  u_1 \otimes D u_1^{s} - 
			u_1^{s+1} \otimes  D + u_1^{s} \otimes  u_1 D \\
			& \sim &  u_1 \otimes u_1^{s} D  - 
			u_1^{s+1} \otimes  D +  s u_1 \otimes u_1^{s-1} u_1 D \\
			& = & (s+1)u_1 \otimes u_1^{s}D - 
			u_1^{s+1} \otimes  D.  \\
		\end{array}$$ }
Therefore $(s+1)u_1 \otimes  u_1^{(s+1)-1}D \sim u_1^{s+1} \otimes D.$  
	By induction the result follows. The proof for case $k_1 < 0$ is made in an analogous way. 
\end{proof}

\begin{proposition} \label{prop-3}
	If $rank([\phi]-I) = n-1$ then the 1-cycle $u_1^{-1} \otimes D$  is not homologous to zero, for any $D \in G.$ 
\end{proposition}
\begin{proof}
We can write $u_1^{-1} \otimes D$ as follows; $u_1^{-1} \otimes u_1 g$, where $g = u_1^{-1}D.$
	It follows from Lemma \ref{lemma-semi} that the semicentralizer $Z(g)$ is isomorphic to $ker([\phi]-I)$ for each $g \in G.$ Since $rank([\phi]-I) = n-1$ then $Z(g) = \{u_1^{s}| s \in \mathbb{Z} \} $ $\cong \mathbb{Z}$. Therefore $H_{1}(Z(g)) \cong \mathbb{Z}$. From \cite[pg 433]{G-N-94}  
	there is a sequence of natural isomorphisms; 
	{\footnotesize
		$$ H_{1}(Z(g)) \to  H_{1}(G, \mathbb{Z}(G/Z(g))) \to H_{1}(G,\mathbb{Z}(C(g))) \to HH_{1}(\mathbb{Z}G, (\mathbb{Z}G)^{\phi})_{C(g)}. $$}
	The class of element $u_1^{s}$ is sent in the class of the 1-cycle $u_1^{s}\otimes u_1^{-s}g$, which is homologous to a 1-cycle; $-s u_1^{-1}\otimes u_1 g $ $=-s(u_1^{-1}\otimes u_1 g)$. 
	Thus, if the 1-cycle was homologous to zero we would have $H_{1}(Z(g)) \cong 0$ which is a contradiction.
\end{proof}

Let us denote by $B_i = u_1^{k_1^{i}} \cdots u_n^{k_n^{i}}$ and 
$D_i = u_1^{l_1^{i}} \cdots u_n^{l_n^{i}}$ elements in $G,$ where $k_{j}^{i}, l_{j}^{i} \in \mathbb{Z}.$ 

\begin{proposition}  \label{prop-2}
	If $rank([\phi]-I) = n-1$ then each 1-cycle $\displaystyle \sum^{t}_{i=1} a_{i} B_i \otimes D_i \in C(\mathbb{Z}G, (\mathbb{Z}G)^{\phi}) $ is homologous to a 1-cycle with the following expression: 
	$\displaystyle \sum^{\bar{t}}_{i=1} \bar{a_{i}} u_1 \otimes D_i^{'}.$
\end{proposition}
\begin{proof}
	Using Propositions \ref{prop-1} and \ref{prop-3}, this is an easy generalization of \cite[Proposition 4.18]{S-20}.	
\end{proof}

\begin{corollary} \label{corollary-2}
	If the cycles $u_1 \otimes D_i $ and $u_1 \otimes D_j$ are 
	in different semiconjugacy classes for all $i \neq j$, $i,j \in \{1, ... , t \}$, then $\displaystyle \sum_{i=1}^{t} u_1 \otimes D_i $ is a nontrivial cycle. Furthermore, $u_1 \otimes D_i$ projects to the same class $[u_1] \in H_{1}(G).$
\end{corollary}



\section{Proof of the main result}

This section shall be devoted to proof Theorem \ref{main-theorem-1}.

\begin{proof}
	Given $H: \mathbb{T}^{n} \times I \to \mathbb{T}^{n}$ by hypothesis there exists a affine homotopy $F$ homotopic to $H$ relative to $\mathbb{T}^{n} \times \{0,1\}.$ We have
 $F \circ ( p_{\mathbb{T}^{n}} \times Id) = p_{\mathbb{T}^{n}} \circ F^{'}$ where 
$$F^{'}(x_1, \cdots, x_n, t) = ( \displaystyle \sum_{j=1}^{n} b_{1j}x_j + c_1t + \epsilon_1 , \cdots , \displaystyle \sum_{j=1}^{n} b_{nj}x_j + c_nt + \epsilon_n ),$$
for some $b_{ij}, c_i \in \mathbb{Z}$ and $ 0 \leq \epsilon_i < 1.$

By \cite[Theorem 1.9 item a]{G-N-94} we have $R(H) = R(F),$ therefore is enough to check the proprieties of Theorem \ref{main-theorem-1} for the homotopy $F.$ Based on that we will compute $R(F).$ From \cite[Theorem 1.9 item c]{G-N-94} we can suppose $H$ cellular. By Proposition \ref{prop-nonzero} is enough to consider the case where $N(F|_{\mathbb{T}^{n}}) = 0,$ and therefore we can assume;

\begin{equation} \label{matrix-phi}
[\phi] = 
	\left( \begin{array}{cccc}
		1 & b_{12} & \cdots & b_{1n}	\\
		0 & b_{22} &  \cdots & b_{2n}   \\
		\vdots & \vdots   &    &  \vdots \\
		0  & b_{n2}	&  \cdots &  b_{nn} \\	
	\end{array}\right).
\end{equation}
	
Note that $w = F(v,I)$ is a loop in $\mathbb{T}^{n}.$ Thus $[w] = u_1^{c_1}u_2^{c_2}...u_n^{c_n},$ for some integers $c_1, c_2, ... , c_n.$ 
We denote by $A$ the following matrix:

\begin{equation} \label{matrix-A}
A = \left( \begin{array}{cccc}
		b_{12} & \cdots & b_{1n} & c_1	\\
		b_{22}-1 &  \cdots & b_{2n}  & c_2 \\
		\vdots   &    &  \vdots  & \vdots  \\
		b_{n2}	&  \cdots &  b_{nn}-1 & c_n \\	
	\end{array}\right).
\end{equation}	
	
	Our proof breaks into two cases. The case $rank(A) = n$ and   $rank(A) < n.$ Firstly we assume $rank(A) = n.$ Note that this  implies $rank([\phi] - I) = n-1.$

Since $\mathbb{T}^{n}$ is a polyhedron, it has a structure of a regular CW-complex. We take an orientation for each k-cell ${E}^{j}_{k}$ in $\mathbb{T}^{n}.$ From \cite[Proposition 4.1]{G-N-94} the trace $R(F)$ is independent of the choice of orientation of cells on $\mathbb{T}^{n}.$ This independence is in terms of homology class. 
	
On the universal covering space $\mathbb{R}^{n}$ we choose a k-cell $\tilde{E}^{j}_{k}$ which projects on ${E}^{j}_{k}.$ We orient $\tilde{E}^{j}_{k}$ compatible with ${E}^{j}_{k}.$ 
We can suppose that $\tilde{E}^{j}_{k}$ is contained in $Y = [0,1] \times \cdots \times [0,1] \subset \mathbb{R}^{n}.$
Considering $C_{\ast}(\mathbb{R}^{n})$ as a right $\mathbb{Z}[\pi_{1}(\mathbb{T}^{n})]$ chain complex as defined in Section \ref{section-2} we have 
	$$ \partial_{i}(e^{i}_{k}) = \displaystyle \sum_{j} [e^{i}_{k}: e^{i-1}_{j}]e^{i-1}_{j} $$ with $[E^{k}_{i}: E^{k-1}_{j}]= 
	[e^{i}_{k}: e^{i-1}_{j}]$ where $[E^{k}_{i}: E^{k-1}_{j}]$ is the incidence of a k-cell $E^{k}_{i}$ to a $(k-1)-$cell. Since that $\mathbb{T}^{n}$ is a regular CW complex then $[E^{k}_{i}: E^{k-1}_{j}] $ belongs to the set $\{0,1,-1\},$ see \cite{W-18}. From definition of the right $\mathbb{Z}G$ action on $C_{\ast}(\mathbb{R}^{n})$ and the fact that each k-cell is contained in $Y$ then, for each $j= 1, ... , n,$ the entries of matrices of operators $\tilde{\partial}_{j}$ will be composed by the following elements: $0, \pm 1, \pm u_i^{-1},$ where $ 1 \leq i \leq n.$ By definition;
	$$ R(F) = tr \left( \begin{array}{ccccc}
		-[\tilde{\partial}_{1}] \otimes [\tilde{D}_{0}] & 0 & 0 & \cdots & 0 \\
		0 & [\tilde{\partial}_{2}] \otimes [\tilde{D}_{1}] & 0 & \cdots & \\
		\vdots &    0 & \ddots &  & \vdots \\
		& \vdots &   &  & 0 \\ 
		0  & 0  &  0 &  & (-1)^{n+1} [\tilde{\partial}_{n}] \otimes [\tilde{D}_{n-1}] \\
	\end{array} \right ),
	$$
	where the elements of matrices $[\tilde{\partial_{j}}]_{ik}$ belong to the set $ \{0, \pm 1, \pm u_i^{-1}\},$ $ 1 \leq i \leq n.$  Thus the general expression of $R(F)$ in  $C_{1}(\mathbb{Z}G, (\mathbb{Z}G)^{\phi})$ would be; 
	{\small \begin{equation} \label{equation-4}
			R(F) =  -1 \otimes (\displaystyle\sum_{j=1}^{m} E_{j} )  + 1 \otimes (\displaystyle\sum_{j=1}^{\bar{m}} D_{j} ) + 
			\displaystyle \sum_{i} \left[ u_i^{-1} \otimes \displaystyle\sum_{j=1}^{n} A_{j}^{i} \right]  - \displaystyle \sum_{i} \left[ u_i^{-1} \otimes \displaystyle \sum_{j=1}^{p} B_{j}^{i} \right],  
	\end{equation} }
	where $E_j, D_j, A_j^{i}, B_j^{i}$ are elements in $G.$
	
	\smallskip
	
If there exists $\overline{F}: \mathbb{T}^{n} \times I \to \mathbb{T}^{n}$ homotopic to $F,$ relative to $\mathbb{T}^{n} \times \{0,1\},$ such that $Fix(\overline{F}) = \emptyset$ then  $R(F)$ is trivial and therefore $L(F) = N(F)= 0.$ From now on, we assume that each homotopy $\overline{F}: \mathbb{T}^{n} \times I \to \mathbb{T}^{n}$ homotopic to $F,$ relative to $\mathbb{T}^{n} \times \{0,1\},$ contains isolated circles in $Fix(\overline{F}).$ The number of these isolated circles for each $\overline{F}$ is finite because $\mathbb{T}^{n}$ is compact.

From Lemma \ref{lemma-unit} each 1-chain $1 \otimes E_j$ is a boundary. Therefore, the 1-chains $ 1 \otimes E_j$ and $-1 \otimes D_j $ are homologous to zero in $C_{1}(\mathbb{Z}G, (\mathbb{Z}G)^{\phi}).$ By Lemma \ref{lemma-1} the 1-chain $u_i^{-1} \otimes A_j^{i}$ is not a cycle for each $ 2 \leq i \leq n.$ Therefore, the 1-chains $u_i^{-1} \otimes A_j^{i}$ and $-u_i^{-1} \otimes B_j^{i},$ for $i \geq 2,$ can not appear in the expression of  $R(F)$ since $R(F)$ is a cycle in $HH_{1}(\mathbb{Z}G, (\mathbb{Z}G)^{\phi}).$

Now let us calculate $Fix(F).$ Note that $F(x,0) = F(x,1).$ Therefore the homotopy $F$ induces a map $\overline{F}:\mathbb{T}^{n+1} = \mathbb{T}^{n} \times \mathbb{S}^{1}  \to \mathbb{T}^{n}$ defined by 
$\overline{F}(x,[t]) = F(x,t).$  Is not too difficult to see that the number of path components in $Fix(F)$ and $Fix(\overline{F})$ are the same. Furthermore $\overline{F}(x,t) = H(x,t) + (\epsilon_1, \cdots , \epsilon_n)$ where $H: \mathbb{T}^{n+1} \to \mathbb{T}^{n}$ is linear and the $\epsilon_i$ are chosen such that $F$ has no fixed point in $\mathbb{T}^{n} \times \{0,1\}.$ The number of path components in $Fix(H)$ is the same as in $Fix(\overline{F}).$

Follows from \cite[Theorem 3.3]{J-01} that the number of path components in $Fix(F)$ is $D([H_{\#}] - [P_{\#}]) = D([\overline{F}_{\#}]-[P_{\#}]) = |det(A)|,$ because the first column of $[\overline{F}_{\#}]-[P_{\#}]$ is null. More precisely $Fix(F)$ is composed by $|det(A)|$ disjoint circles. In fact, we have $Fix(F) = p_{\mathbb{T}^{n}}((F^{'}-P)^{-1}(\mathbb{Z}^{n})).$ Follows from expression of $F^{'}$ that a point $z = (x_1, \cdots, x_n, t)$ belongs to the set $Fix(F^{'})$ if and only if $z$ is a solution of the following system:

\begin{equation} \label{system-solution}			
 \left( \begin{array}{ccccc}
 0 & b_{12} & \cdots & b_{1n} & c_1	\\
 0 &	b_{22}-1 &  \cdots & b_{2n}  & c_2 \\
	\vdots   &    &  \vdots  & \vdots  \\
 0 & b_{n2}	&  \cdots &  b_{nn}-1 & c_n \\	
\end{array}\right) \left( \begin{array}{c}
x_1	\\
x_2  \\
\vdots   \\
x_n \\
t \\	
\end{array}\right)  = \left( \begin{array}{c}
-\epsilon_1 + l_1	\\
-\epsilon_2 + l_2 \\
\vdots   \\
-\epsilon_{n-1} + l_{n-1} \\
-\epsilon_n + l_n\\	
\end{array}\right)
\end{equation}
for some $l_1, \cdots, l_n \in \mathbb{Z}.$
Note that in the System \eqref{system-solution} we have $0 \leq x_1 \leq 1,$ which implies that each path component in $Fix(F)$ is a  isolated circle. 

Note that two different circles belong to the different fixed point classes. In fact, let $C_1 = (x_1, x^{1}_{2}, \cdots, x^{1}_{n}, t^{1})$ and $ C_2 = (x_1, x^{2}_{2}, \cdots, x^{2}_{n}, t^{2})$ two circles in $Fix(F).$ Each fixed point class has the form  $p_{\mathbb{T}^{n}}(Fix(\widetilde{F}))$ where $\widetilde{F}$ is a lift of $F.$    We have $C_1 = \widetilde{F}(C_1),$ $C_2 = \widetilde{F}(C_2).$  Note that  each lift of $F$ has the form $F^{'}(x,t) + (m_1, \cdots, m_n)$ where $m_j \in \mathbb{Z}.$   Thus $C_1 - C_2 = \widetilde{F}(C_1-C_2)$ which implies $(\widetilde{F}-P)(C_1-C_2) = (0, \cdots, 0)$ where $P: \mathbb{R}^{n} \times I \to \mathbb{R}^{n}$ is the projection. As  $rank(A) = n$ then we must have $x^{1}_{2} = x^{2}_{2}, \cdots, x^{1}_{n} = x^{2}_{n}, t^{1} = t^{2}, $ and therefore $C_1 = C_2.$

The cardinality of the semiconjugacy classes $G_{\phi}$ is equal to $|det(A)|,$ that is, is the same as the number of isolated circles in $Fix(F) \cap (0,1).$ In fact, by Lemma \ref{lemma-se} the cardinality of the set $G_{\phi}$ is given by: $\#(coker(\phi-P_{\#})).$   
We have $[w] = u_1^{c_1}u_2^{c_2}...u_n^{c_n}$ 
for some integers $c_1, c_2, ... , c_n.$ Therefore the image of $(\phi-P_{\#})$ in $\pi_{1}(\mathbb{T}^{n})$ is generated by columns of the following matrices:
$$[\phi] - [P_{\#}] = 
\left( \begin{array}{cccc}
	0 & b_{12} & \cdots & b_{1n}	\\
	0 & b_{22}-1 &  \cdots & b_{2n}   \\
	\vdots & \vdots   &    &  \vdots \\
	0  & b_{n2}	&  \cdots &  b_{nn}-1 \\	
\end{array}\right) \,\,\,\,\, and \,\,\,\,\,
\left( \begin{array}{c}
	c_1	\\
	c_2 \\
	\vdots  \\
	c_n \\	
\end{array}\right),$$
that is, the image of $(\phi-P_{\#})$ is generated by the columns of matrix $A$ where $A$ is given by:
$$A = \left( \begin{array}{cccc}
	b_{12} & \cdots & b_{1n} & c_1	\\
	b_{22}-1 &  \cdots & b_{2n}  & c_2 \\
	\vdots   &    &  \vdots  & \vdots  \\
	b_{n2}	&  \cdots &  b_{nn}-1 & c_n \\	
\end{array}\right).$$

From hypothesis we have $rank(A) = n.$ Therefore $\#coker(\phi-P_{\#}) = 
\#(\pi_1(\mathbb{T}^{n})/im(\phi-P_{\#})) = \# (\mathbb{Z}^{n}/A(\mathbb{Z}^{n})) = 
|det(A)|$ since $A$ is non-singular.

Since each circle in $Fix(F^{'})$ is parallel to the axis of $x_1$ then choosing an orientation for a circle all the others will have the same orientation according to the orientation of these circles defined in \cite{D-G-90}. But the orientation of the circles in $Fix(F)$ is compatible with the orientation of the circles in $Fix(F^{'}).$ Thus, all circles in $Fix(F)$ have the same orientation and therefore all cycles in $R(F)$ will have the same signal. From these facts, the one-parameter trace of $F$ will have the following  expression in  $HH_{1}(\mathbb{Z}G,(\mathbb{Z}G)^{\phi}):$

\begin{equation} \label{equation-5}
		R(F) =   u_1^{-1} \otimes \displaystyle\sum_{j=1}^{m} A_{j} 
\end{equation} or 
\begin{equation} \label{equation-6}
		R(F) =    - u_1^{-1} \otimes \displaystyle\sum_{j=1}^{m} B_{j} 
\end{equation} 
where $A_{j},$ $B_{j}$ are elements in $G.$ Let us consider the  expression \eqref{equation-5}. The proof using the expression  \eqref{equation-6} is analogous.

From  Proposition \ref{prop-3} each 1-cycle $u_1^{-1} \otimes A_{j} $ is non trivial and thus represents a nonzero C-component. Furthermore, by definition gives in \cite[Page 434]{G-N-94} each circle in $Fix(F)$ contributes to only one element in $R(F).$  
Thus, each semiconjugacy class contains only one cycle $u_1^{-1} \otimes A_{j} $ in $R(F),$ and therefore the one-parameter Nielsen number of $F$ is; $$N(F) = m = |det(A)|.$$

From Section \ref{section-2}, the one-parameter Lefschetz class is the image of $R(F)$ in $H_{1}(\pi_{1}(\mathbb{T}^{n}),\mathbb{Z})$ by homomorphism induced by  inclusion $i: Z(g_{C}) \to \pi_{1}(\mathbb{T}^{n})$.
By Proposition \ref{prop-3} each cycle $u_1^{-1} \otimes  A_{j} $ will be sent in the same class $-[u_1].$  
Therefore the image of $R(F)$ in $H_{1}(\pi_{1}(\mathbb{T}^{n}),\mathbb{Z})$ is;  
$$ L(F) = \displaystyle\sum_{i=1}^{m} -[u_1] = -m[u_1] = -N(F)[u_1].$$
Thus $L(F) = N(F) \alpha,$ where $\alpha = [u_1]$ or $-[u_1].$

Now we assume $rank(A) < n.$ In this case we have $im(\phi-P_{\#}) \subsetneq \mathbb{Z}^{n}.$ Let $w_0 \notin im(\phi-P_{\#}).$ 
	Define $\widetilde{F}: \mathbb{T}^{n} \times I \to \mathbb{T}^{n}$ by $\widetilde{F}(x,t) = F(x,t) + w_0sin(2t\pi).$ The map $Q: \mathbb{T}^{n} \times I \times I \to \mathbb{T}^{n}$ define by 
	$Q(x,t,s) = F(x,t)+ sw_0sin(2t\pi)$ is a homotopy between $F$ and $\widetilde{H}$ relative to $\mathbb{T}^{n} \times \{0,1\}.$ Since  $w_0 \notin im(\phi-P_{\#})$ then  there are no circles in $Fix(\widetilde{H}) \cap (\mathbb{T}^{n} \times (0,1)).$ Therefore $R(\widetilde{F}) = 0$ which implies $R(F) = 0,$ $N(F) = 0$  and $L(F) = 0.$ 
\end{proof}


\section{Applications}

In this section we present some applications of Theorems \ref{main-theorem-1}  for compute the minimum number of path components   in the fixed point set of some maps.

\

{\bf I.} Let $X$ be a finite CW complex and $F:  X  \times I \to X$ be a homotopy such that $F(x,0) = F(x,1).$ For example, when $X = \mathbb{T}^{n},$ all linear homotopies satisfies $F(x,0) = F(x,1),$ because $F(x,1) = F(x,0)+(d_1, ..., d_n),$ where $d_1, ..., d_n$ are integer numbers. 
Denote $S^{1} = \displaystyle \frac{I}{0 \sim 1}.$  The homotopy $F$ induces a map $\overline{F}: X \times S^{1} \to X$ defined by 
$$\overline{F}(x,[t]) = F(x,t).$$ 

It is not difficult to see that each homotopy 
$H: X \times I \times I \to X$ from $F$ to a map $F^{'}$ relative to 
$X \times \{0,1\}$ is equivalent to a homotopy $\overline{H}: X \times S^{1} \times I \to X$ from $\overline{F}$ to $\overline{F^{'}}$ relative to $(v,[0]).$ 
If $F$ has no fixed points in $X \times \{0,1\}$ then we must have 
$N(F|_{X}) = 0,$ and the minimum number of path components in $Fix(F)$ and $Fix(\overline{F}),$ as $F$ runs over a homotopy class of maps 
$X\times I \to X$ relative to $X\times \{0,1\}$, must coincide. 

Let us consider $X = \mathbb{T}^{n}$ and $F$ a affine homotopy.  Suppose that $N(F|_{\mathbb{T}^{n}}) = 0.$   In this case the one-parameter  Nielsen number of $F$ given in Theorem \ref{main-theorem-1} coincides with the invariant $D([{\overline{F}}_{\#}] - [{\overline{P}}_{\#}])$ presented in \cite[Theorem 3.3]{J-01}, where $P$ is the projection and the matrix of $F_{\#}$ is as in Theorem \ref{main-theorem-1}. In fact, from \cite{J-01}  $D([{\overline{F}}_{\#}] - [{\overline{P}}_{\#}])$ is defined by
$$ D([{\overline{F}}_{\#}] - [{\overline{P}}_{\#}]) = gcd\{([{\overline{F}}_{\#}] - [{\overline{P}}_{\#}])_{\alpha_i}, \,\,\,\,\,\, 1 \leq i \leq n+1 \},$$
where $([{\overline{F}}_{\#}] - [{\overline{P}}_{\#}])_{\alpha_i}$ denotes the determinant of the matrix $[{\overline{F}}_{\#}] - [{\overline{P}}_{\#}]$ with the column $\alpha_i$ removed. In our case we have;
$$[{\overline{F}}_{\#}] - [{\overline{P}}_{\#}] = 
\left( \begin{array}{ccccc}
	0 & b_{12} & \cdots & b_{1n} & c_1	\\
	0 & b_{22}-1 &  \cdots & b_{2n} & c_2  \\
	\vdots & \vdots   &    &  \vdots & \vdots \\
	0  & b_{n2}	&  \cdots &  b_{nn}-1 & c_n \\	
\end{array}\right).$$

Since the first column of the above matrix is zero then 
$$ D([{\overline{F}}_{\#}] - [{\overline{P}}_{\#}]) = |det(A)| = N(F),$$
where $A$ is as in Theorem \ref{main-theorem-1}.
In this case the affine homotopies ``realize'' the one-parameter Nielsen number.     

In which case where  $N(F|_{\mathbb{T}^{n}}) \neq 0$ the Proposition \ref{prop-nonzero} guarantees that the one-parameter Nielsen number $N(F)$ is zero. But in this case we have $D([{\overline{F}}_{\#}] - [{\overline{P}}_{\#}])  \neq 0.$ This happens because arcs connecting $\mathbb{T}^{n} \times \{0\}$ to $\mathbb{T}^{n} \times \{1\}$ in $Fix(F)$ will produce circles in $Fix(\overline{F}).$

\

{\bf II.} Let $M$ be a fiber bundle with base $S^{1}$ and fiber $\mathbb{T}^{n}.$ The total space $M$ is given by
$$M = M(h) = \displaystyle \frac{\mathbb{T}^{n} \times I}{(z,0) \sim (h(z),1)}$$ 
where $h$ is a homeomorphism of $\mathbb{T}^{n}.$ For more details in the cases $n=1$ or $n=2$ see \cite{G-87} and \cite{G-P-V-04}. The projection map $p: M(h) \to S^{1} = I/ 0 \simeq 1$ is given by $p(<z,t>) = <t>,$ where $<z,t>$ denotes the class of $(z,t)$ in $M(h).$  
Therefore each map $f: M(h)  \to M(h),$ over $S^{1},$ is given by $$f(<z,t>) = <F(z,t), t>,$$ 
where $F: \mathbb{T}^{n} \times I \to \mathbb{T}^{n}$ is a homotopy. The term over $S^{1}$ means $p \circ f = p.$ By the long exact sequence of the fibration $p$ in homotopy we obtain;
$$ \pi_1(M(h),0) \approx \pi_1(\mathbb{T}^{n}) \rtimes \pi_1(S^{1}).$$
We denote the generators of $\pi_1(M(h),0)$ by $u_1, \cdots, u_n , d.$

We are interested to use the Theorem \ref{main-theorem-1} to compute the minimal path components of $Fix(f).$ This phrases means, we want to find a map $g$ fiberwise homotopic to $f$ such that the path components in $Fix(g)$ is minimal. 

Given $f: M(h) \to M(h)$ we will assume that $N(f|_{\mathbb{T}^{n}}) = 0.$ 
Let us take $G: \mathbb{T}^{n} \times I \to \mathbb{T}^{n}$ defined by 
$$G(x_1, \cdots, x_n, t) = ( \displaystyle \sum_{j=1}^{n} b_{1j}x_j + c_1t + \epsilon_1 , \cdots , \displaystyle \sum_{j=1}^{n} b_{nj}x_j + c_nt + \epsilon_n ),$$
where $[\phi] = [(H|_{\mathbb{T}^{n}})_{\#}] = (b_{ij}) = [(f|_{\mathbb{T}^{n}})_{\#}]$ and $c_i$ are given by $[f(<v,t>)] = u_1^{c_1} \cdots u_n^{c_n} d. $  Since $N(f|_{\mathbb{T}^{n}}) = 0$ we can suppose that $[\phi]$ has the same expression as in \eqref{matrix-phi}.

We consider $\mathcal{H}$ the set of all homeomorphism $h$  of $\mathbb{T}^{n}$ such that $G(z,0) = G(h(z),1).$ In our application we consider $M(h)$ with $h \in \mathcal{H}.$ If $h \in \mathcal{H}$ then $G$ induces a fiber map $g: M(h) \to M(h)$ over $S^{1}$ defined by $g(<z,t>) = <G(z,t),t>,$ for an example of this in case $n=2$ see \cite{S-S-17}.

By construction we have $g_{\#} = f_{\#}.$ Since $M(h)$ is $K(\pi, 1)$ then $g$ is homotopic to $f$ by a homotopy over $S^{1}.$
We can choose $ 0 \leq \epsilon_i < 1$ such that $G$ has no fixed points for $t=0,1,$ which implies that $g$ also has no fixed points for $t = 0,1.$ Therefore $Fix(g) \simeq Fix(G).$

Note that $G(z,0) = G(z,1)$ in $\mathbb{T}^{n},$ where $z = (x_1, \cdots, x_n).$ In this case each homotopy of $G$ relative to $\mathbb{T}^{n} \times \{0,1\}$ is equivalent to a homotopy over $S^{1}$ of $g$ relative to $(v,[0]).$ Therefore to minimize the path components of $Fix(g)$ by homotopies over $S^{1}$ relative to $(v, [0])$ is equivalent to minimize the path components of $Fix(G)$ by homotopies relative to $\mathbb{T}^{n} \times \{0,1\}.$

The one-parameter Nielsen number $N(G)$ is a lower bound for the number of path components in $Fix(H)$ for each $H$ homotopic to $G$ relative to $\mathbb{T}^{n}\times \{0,1\}.$    
Follows from Theorem \ref{main-theorem-1} that the minimum number of path components, or the minimum number of circles, in $Fix(G)$ is given by $N(G) = |det(A)|,$ where $A$ is as in \eqref{matrix-A}, and therefore the minimum number of path components in $Fix(f)$ is $|det(A)|.$

In case $n=2$ the number $|det(A)|$ appeared in the main theorem of \cite{G-P-V-04} to say only when a fiber map $f : M(h) \to M(h)$ can be deformed or not to a fixed point free map over $S^{1}.$ Here we give a complete description because we proved that $|det(A)|$ is the minimum number of path components in $Fix(f)$ up to deformations over $S^{1}.$  Therefore $f$ can be deformed to a fixed point free map over $S^{1}$ if and only if $|det(A)| = 0$ in the case where $h \in \mathcal{H}.$



\end{document}